\documentclass[11pt]{article}

\usepackage[hidelinks]{hyperref}
\hypersetup{
  colorlinks   = true, 
  urlcolor     = blue, 
  linkcolor    = blue, 
  citecolor   = red 
}
\usepackage{amsmath,amsthm,amssymb}
\usepackage{graphicx}
\usepackage{bbm}
\newcommand{\sy}{\boldsymbol{\Psi}}
\newcommand{\py}{\boldsymbol{\Phi}}
\newcommand{\T}{\mathbb{T}}

\usepackage{xcolor}
\usepackage[margin=1in]{geometry} 
\usepackage{enumerate,mathtools,mathrsfs,bm,graphicx}
\usepackage[numbers, super]{natbib}
\usepackage[hidelinks]{hyperref}
\newcommand{\N}{\mathbb{N}}									

\newcommand{\R}{\mathbb{R}}

\newcommand{\vertiii}[1]{{\left\vert\kern-0.25ex\left\vert\kern-0.25ex\left\vert #1 
    \right\vert\kern-0.25ex\right\vert\kern-0.25ex\right\vert}}
    
\newcommand{\overbar}[1]{\mkern 1.5mu\overline{\mkern-1.5mu#1\mkern-1.5mu}\mkern 1.5mu}
\newcommand{\inner}[2]{\left\langle #1, #2 \right\rangle}
\newcommand{\norm}[1]{\left\Vert #1 \right\Vert}

\newtheorem{theorem}{Theorem}[section]

\newtheorem{lemma}[theorem]{Lemma}

\newtheorem{definition}[theorem]{Definition}

\begin{document}
	\title{It\^{o}-Stratonovich Conversion in Infinite Dimensions for Unbounded, Time-Dependent, Nonlinear Operators}
	\author{Daniel Goodair \footnote{\'{E}cole Polytechnique F\'{e}d\'{e}rale de Lausanne, daniel.goodair@epfl.ch}}
	\date{\today} 
	\maketitle
\setcitestyle{numbers}	
\thispagestyle{empty}
\begin{abstract}
We prove that a solution, in a variational framework, to the Stratonovich stochastic partial differential equation with noise $\mathcal{G}\left(t, \sy_t\right) \circ d\mathcal{W}_t$ is given by a solution to the It\^{o} equation with It\^{o}-Stratonovich corrector $\frac{1}{2}\sum_{i=1}^\infty D_u\mathcal{G}_i\left(t, \sy_t\right)\left[\mathcal{G}_i(t,\sy_t)\right]dt$. Here $\mathcal{G}_i$ denotes the action of $\mathcal{G}$ on the $i^{\textnormal{th}}$ component of the cylindrical noise, and $D_u\mathcal{G}_i$ its Fr\'{e}chet partial derivative in the Hilbert space for which the It\^{o} form is satisfied. The noise operator $\mathcal{G}$ may be time-dependent, nonlinear, and unbounded in the sense of differential operators; in the latter case, one must pass to a larger space in order to solve the Stratonovich equation. Our proof relies on martingale techniques, and the results apply to fluid equations with time-dependent and nonlinear transport noise.

\end{abstract}

    
\tableofcontents
\textcolor{white}{Hello}
\thispagestyle{empty}
\newpage

\setcounter{page}{1}

\section{Introduction} \label{section introduction}

Whilst Stratonovich stochastic partial differential equations (SPDEs) arise plentifully from various physical principles, the Stratonovich integral lacks key favourable properties of its It\^{o} counterpart. Therefore, when tasked with obtaining a solution to a Stratonovich SPDE, one would rather work with a corrected It\^{o} equation whose solution is known to solve the original Stratonovich SPDE. The goal of this paper is to rigorously provide such a framework, applicable to noise operators which may be nonlinear, time-dependent, and mapping to a larger Hilbert Space. In analogy with classical SDE theory, the It\^{o}-Stratonovich corrector is given by the Fr\'{e}chet partial derivative of the noise operator acting on itself.

\subsection{Main Result}
We give the set-up and statement of the main result. We fix a time horizon $[0,T]$ for $T > 0$, a filtered probability space $(\Omega,\mathcal{F},(\mathcal{F}_t), \mathbbm{P})$ satisfying the usual conditions of completeness and right continuity, supporting a Cylindrical Brownian Motion $\mathcal{W}$ over some separable Hilbert Space $\mathfrak{U}$ with orthonormal basis $(e_i)$. Let $$V \hookrightarrow H \hookrightarrow U \hookrightarrow X$$ be a quartet of embedded separable Hilbert Spaces, and introduce the measurable mapping
$$\mathcal{A}: [0,T] \times V \rightarrow U$$
whereby there exists constants $c,p$ such that for all $t \in [0,T]$ and $\phi \in V$,
$$\norm{\mathcal{A}(t,\phi)}_U \leq c\left(1 + \norm{\phi}_H^p\right)\left(1 + \norm{\phi}_V^2\right).$$ Furthermore we understand $\mathcal{G}$ as a measurable mapping
\begin{align*}
   \mathcal{G}:[0,T] \times U \rightarrow \mathscr{L}^2(\mathfrak{U};X), \qquad   \mathcal{G}|_H: [0,T] \times H \rightarrow \mathscr{L}^2(\mathfrak{U};U), \qquad \mathcal{G}|_V: [0,T] \times V \rightarrow \mathscr{L}^2(\mathfrak{U};H)
\end{align*}
where $\mathscr{L}^2\left(\mathfrak{U};X \right)$ denotes the space of Hilbert-Schmidt Operators from $\mathfrak{U}$ into $X$. In fact we define $\mathcal{G}$ over $\mathfrak{U}$ by its action on the basis vectors $\mathcal{G}(s, \phi, e_i) = \mathcal{G}_i(s, \phi)$. We assume that for each $i$, $\mathcal{G}_i:[0,T] \times U \rightarrow X$ is continuous with Fr\'{e}chet partial derivatives $\partial_t\mathcal{G}_i$, $D_u\mathcal{G}_i$, $D_{uu}\mathcal{G}_i$ continuous and bounded on bounded subsets of $[0,T] \times U$. Moreover, let us assume that there exists constants $c_i, q$ such that for all $t \in [0,T]$, $\phi \in V$, $\psi \in H$ that
$$\norm{\mathcal{G}_i(t,\phi)}_H \leq c_i\left(1 + \norm{\phi}_V\right), \qquad \norm{D_u\mathcal{G}_i(t,\psi)}_{\mathscr{L}(H;U)} \leq c_i\left(1 + \norm{\psi}_H^q\right), \qquad \sum_{i=1}^\infty c_i^2 < \infty$$
where $\mathscr{L}(H;U)$ denotes the space of bounded linear operators from $H$ into $U$, equipped with operator norm $\norm{\cdot}_{\mathscr{L}(H;U)}$. Under these assumptions, we can now state the main theorem.

\begin{theorem} \label{main theorem}
Let $\sy_0:\Omega \rightarrow H$ be $\mathcal{F}_0-$measurable, alongside a pair $(\sy,\tau)$ comprised of a stopping time $\tau \in (0,T]$ $\mathbbm{P}-a.s.$ and a process $\sy$ such that for $\mathbbm{P}-a.e.$ $\omega$, $\sy_{\cdot}(\omega) \in C\left([0,T];H\right)$ and $\sy_{\cdot}(\omega)\mathbbm{1}_{\cdot \leq \tau(\omega)} \in L^2\left([0,T];V\right)$ with $\sy_{\cdot}\mathbbm{1}_{\cdot \leq \tau}$ progressively measurable\footnote{We really mean a version of the process, see [\cite{goodair2024stochastic}] Remark 3.1.} in $V$, satisfying the identity
\begin{equation} \label{converted identity}    \sy_{t} = \sy_0 + \int_0^{t\wedge \tau} \left(\mathcal{A}\left(s,\sy_s\right) + \frac{1}{2}\sum_{i=1}^\infty D_u\mathcal{G}_i\left(s, \sy_s\right)\left[\mathcal{G}_i(s,\sy_s)\right]\right)ds + \int_0^{t \wedge \tau}\mathcal{G}\left(s, \sy_s\right) d\mathcal{W}_s
\end{equation}
$\mathbbm{P}-a.s.$ in $U$ for all $t \geq 0$. Then $\sy$ satisfies the identity
\begin{equation} \label{strat ident}    \sy_{t} = \sy_0 + \int_0^{t\wedge \tau} \mathcal{A}\left(s,\sy_s\right)ds + \int_0^{t \wedge \tau}\mathcal{G}\left(s, \sy_s\right) \circ d\mathcal{W}_s
\end{equation}
$\mathbbm{P}-a.s.$ in $X$ for all $t \in [0,T]$.

\end{theorem}

Theorem \ref{main theorem} will be proven in Section \ref{section proof of main}.

\subsection{Motivation and Relation to the Literature} \label{subs motive}

Of the many physically relevant Stratonovich SPDEs, let us first draw attention to the thriving literature of transport noise in fluid mechanics. Considered since the work of Kraichnan [\cite{kraichnan1968small}], transport noise in fluids have attracted significant interest in the last ten years due to a plethora of modelling and theoretical developments. Briefly, these include the use of geometric variational principles [\cite{holm2015variational}], a Lagrangian Reynolds Decomposition and Transport Theorem [\cite{memin2014fluid}], stochastic model reduction [\cite{debussche2024second}, \cite{flandoli20212d}] and regularisation by noise [\cite{flandoli2021delayed}, \cite{flandoli2021high}]. Whilst the type of transport noise varies based on equation and philosophy, the shared term is a Stratonovich integral $$\sum_{i=1}^\infty \int_0^t\mathcal{L}_{\xi_i}\sy_s \circ dW^i_s$$
where $\mathcal{L}_{\xi_i}$ is defined for sufficiently regular $f: \mathscr{O} \rightarrow \R^d$, $\mathscr{O} \subset \R^n$ by $\mathcal{L}_{\xi_i}f = \sum_{j=1}^n\xi_i^j\partial_jf$. Each $\xi_i:\mathscr{O} \rightarrow \R^n$ and the collection $(\xi_i)$ will either have strong decay or orthogonality so that the sum converges, whilst $(W^i)$ are independent standard Brownian Motions. The process $\sy$ is prescribed by an evolution equation, for which it is essential to rewrite the Stratonovich integral as a corrected It\^{o} one; the prevalent approach has been a heuristic observation, see for example [\cite{alonso2020well}, \cite{butori2024ito},  \cite{butori2024mean}, \cite{crisan2019solution}, \cite{flandoli20242d}, \cite{galeati2020convergence}], arguing by linearity of $\mathcal{L}_{\xi_i}$ and standard stochastic calculus that the It\^{o}-Stratonovich corrector for each term in the summand takes the form
\begin{align*}
    \frac{1}{2}\left[\mathcal{L}_{\xi_i}\sy, W^i \right]_t &= \frac{1}{2}\mathcal{L}_{\xi_i}\left[\sy, W^i \right]_t =  \frac{1}{2}\mathcal{L}_{\xi_i}\left[\sum_{j=1}^\infty\int_0^\cdot \mathcal{L}_{\xi_j}\sy_s dW^j, W^i \right]_t\\
   &=  \frac{1}{2}\mathcal{L}_{\xi_i}\left(\int_0^t \mathcal{L}_{\xi_i}\sy_s ds \right) = \frac{1}{2}\int_0^t \mathcal{L}_{\xi_i}^2\sy_s ds.
\end{align*}
Innumerably many more works cite one of the above as justification, or directly state the It\^{o} form given the now extensive literature on these equations. The authors proceed with their analysis and results on the It\^{o} form, without rigorously addressing how these results apply to the original and well-motivated Stratonovich SPDE. A key consideration is the fact that the operator $\mathcal{L}_{\xi_i}$ is unbounded, as a differential operator, so in the above heuristic one is really passing to a larger function space when putting $\mathcal{L}_{\xi_i}$ inside of the cross-variation or stochastic integral. More apparently, whilst $\sy$ would then be specified by a second order SPDE, the evolution equation for $\mathcal{L}_{\xi_i}\sy$ takes three derivatives of $\sy$ so the drift terms in $\mathcal{L}_{\xi_i}\sy$ may no longer be of finite variation in the function space which $\sy$ satisfies its SPDE.\\

At a glance, one could avoid this issue of regularity by arguing in weak form. There, one may pass the derivatives of $\mathcal{L}_{\xi_i}$ onto test functions which can bare the regularity requirement: this approach was given in [\cite{flandoli2021high}, \cite{galeati2023weak}], where only one dimensional martingale arguments are needed as the tested identity is satisfied in $\R$. In reconciling this notion of weak solution, one meets the issue of identifying
$$ \int_0^t \inner{\mathcal{L}_{\xi_i}\sy}{\phi} \circ dW^i_s = \inner{\int_0^t \mathcal{L}_{\xi_i}\sy \circ dW^i_s}{\phi}$$
whereby the cost of a derivative becomes key once more, and infinite dimensional martingale arguments are required.\\

These heuristics were made rigorous by the author and Dan Crisan in the book [\cite{goodair2024stochastic}], examining properties of the cross-variation in Hilbert Spaces and quantifying the `loss of a derivative' experienced in the conversion, understood broadly in the variational framework by an analytically strong solution of the It\^{o} equation is an analytically weak solution of the Stratonovich SPDE. However, following the above heuristics, the arguments relied heavily on the fact that the noise operator was linear and time-independent. Furthermore, the resulting expression was not obviously compatible with the well-known conversion in finite dimensions. Sufficiency for applications and mathematical completeness thus motivated Theorem \ref{main theorem}.\\

In Section \ref{section applications} we shall consider applications of Theorem \ref{main theorem} to relevant cases of time-dependence and nonlinearity, which we comment on now. Whilst there are many such Stratonovich noise structures available, we continue our motivations with transport noise. There is a strong argument, from the modelling perspective, to include time-dependence in the spatial correlation functions $(\xi_i)$. In the words of Holm's breakthrough paper, the $(\xi_i)$ are `specified from the physics of the problem \dots obtained from, say, coarse-grained observations or computations'; it is well reported that the ocean exhibits \textit{memory}, see for example [\cite{shi2022global}, \cite{woods1981memory}], so the observed spatial correlations of ocean dynamics should evolve over time. An existence result for the stochastic Navier-Stokes equations is given in Subsection \ref{subs time dep}. In the direction of nonlinearity, a transport noise was proposed in [\cite{flandoli20242d}] following the idea that `turbulence is more developed in regions of high large-scale vorticity; hence, the small-scale noise should be modulated by an increasing function [of vorticity]'. The small-scale noise produces a large-scale transport noise in the separation of scales limit, given in [\cite{flandoli20242d}] by
$$\sum_{i=1}^\infty \int_0^t\mathcal{L}_{\xi_i}\left(f(w_s)\right) \circ dW^i_s $$
where $w: \T^2 \rightarrow \R$ is the fluid vorticity and $f: \R \rightarrow \R$ modulates the noise intensity. The authors formally calculate the It\^{o}-Stratonovich corrector by firstly applying the above heuristic, and secondly using the standard one-dimensional It\^{o} Formula for $f(w(x))$ applied pointwise in space. More generally, one could consider the noise
$$\sum_{i=1}^\infty \int_0^t\mathcal{L}_{\xi_i}F(\sy_s) \circ dW^i_s $$
where $F$ is a function space valued mapping with genuinely infinite dimensional nonlinear effects. Whilst the prior heuristic no longer applies, we compute the corrector in Subsection \ref{subs nonlinear trans} by applying Theorem \ref{main theorem} and show that it agrees with the specific choice of $F$ as an evaluation map $F(\psi)(x) = f(\psi(x))$.\\

More generally, we are only aware of two results concerning an It\^{o}-Stratonovich conversion in infinite dimensions: the paper [\cite{twardowska2004relation}] and the book [\cite{duan2014effective}] Subchapter 4.5.2. In both cases, the noise operator must be time-independent and mapping within the same Hilbert Space. Mild solutions enjoying moment estimates are considered in [\cite{twardowska2004relation}] whilst the identification is only formal in [\cite{duan2014effective}]. Critically, both define the Stratonovich integral by, and conduct their proofs using, the limit of a sum over partitions in time evaluated at the mid-point of the intervals; we use the semi-martingale definition and approach. A rigorous conversion in a variational framework, where the noise operator is time-dependent and unbounded, understood via martingale techniques, all appear to be novelties of this work.

\subsection{Preliminaries} \label{stochastic prelims}

In this subsection we establish the necessary prerequisites to understand and prove Theorem \ref{main theorem}. Recall that we have fixed a time horizon $[0,T]$ and a filtered probability space $(\Omega,\mathcal{F},(\mathcal{F}_t), \mathbbm{P})$ satisfying the usual conditions of completeness and right continuity, supporting a Cylindrical Brownian Motion $\mathcal{W}$ over some separable Hilbert Space $\mathfrak{U}$ with orthonormal basis $(e_i)$. Recall (e.g. [\cite{lototsky2017stochastic}], Definition 3.2.36) that $\mathcal{W}$ admits the representation $\mathcal{W}_t = \sum_{i=1}^\infty e_iW^i_t$ as a limit in $L^2(\Omega;\mathfrak{U}')$ whereby the $(W^i)$ are a collection of i.i.d. standard real valued Brownian Motions and $\mathfrak{U}'$ is an enlargement of the Hilbert Space $\mathfrak{U}$ such that the embedding $J: \mathfrak{U} \rightarrow \mathfrak{U}'$ is Hilbert-Schmidt and $\mathcal{W}$ is a $JJ^*-$cylindrical Brownian Motion over $\mathfrak{U}'$. Let $\mathcal{H}, \mathcal{K}$ denote separable Hilbert Spaces. Given a process $B:[0,T] \times \Omega \rightarrow \mathscr{L}^2(\mathfrak{U};\mathcal{H})$ progressively measurable and such that $B \in L^2\left(\Omega \times [0,T];\mathscr{L}^2(\mathfrak{U};\mathcal{H})\right)$, for any $0 \leq t \leq T$ we define the stochastic integral $$\int_0^tB_sd\mathcal{W}_s:=\sum_{i=1}^\infty \int_0^tB_s(e_i)dW^i_s,$$ where the infinite sum is taken in $L^2(\Omega;\mathcal{H})$. We can extend this notion to processes $B$ which are such that $B(\omega) \in L^2\left( [0,T];\mathscr{L}^2(\mathfrak{U};\mathcal{H})\right)$ for $\mathbbm{P}-a.e.$ $\omega$ via the traditional localisation procedure. In this case the stochastic integral is a local martingale in $\mathcal{H}$. We defer to [\cite{goodair2024stochastic}] Chapter 2 for further details on this construction and properties of the stochastic integral.\\

To prove Theorem \ref{main theorem} we shall make use of an infinite dimensional It\^{o} Formula. Recall that $F:[0,T] \times \mathcal{H} \rightarrow \mathcal{K}$ has Fr\'{e}chet partial derivatives, if they exist, as mappings:
\begin{itemize}
    \item $\partial_tF: [0,T] \times \mathcal{H} \rightarrow \mathcal{K}$;
    \item $D_hF: [0,T] \times \mathcal{H} \rightarrow \mathscr{L}\left( \mathcal{H};\mathcal{K}\right)$;
    \item $D_{hh}F: [0,T] \times \mathcal{H} \rightarrow \mathscr{L}\left( \mathcal{H};\mathscr{L}\left(\mathcal{H};\mathcal{K}\right)\right)$.
\end{itemize}

Note that strictly $D_tF: [0,T] \times \mathcal{H} \rightarrow \mathscr{L}\left([0,T];\mathcal{K} \right)$ which we identify with $\mathcal{K}$, in agreement with the usual derivative $\partial_tF$ as written. The following can be found in [\cite{curtain1970ito}] Theorem 3.8, with minor modifications taken from [\cite{gawarecki2010stochastic}] Theorem 2.9 to ensure correctness and remove the assumption on moment estimates.

\begin{lemma}\label{Ito formula}
    Suppose that:
    \begin{itemize}
    \item $\py_0 : \Omega \rightarrow \mathcal{H}$ is $\mathcal{F}_0-$measurable;
        \item $\phi: [0,T] \times \Omega \rightarrow \mathcal{H}$ is adapted and belongs $\mathbbm{P}-a.s.$ to $L^1\left([0,T];\mathcal{H}\right)$;
        \item $B: [0,T] \times \Omega \rightarrow \mathscr{L}^2\left(\mathfrak{U};\mathcal{H}\right)$ is progressively measurable and belongs $\mathbbm{P}-a.s.$ to $L^2\left([0,T]; \mathscr{L}^2\left(\mathfrak{U};\mathcal{H}\right)  \right)$;
        \item $F:[0,T] \times \mathcal{H} \rightarrow U$ is continuous with Fr\'{e}chet partial derivatives $\partial_tF$, $D_hF$, $D_{hh}F$ continuous and bounded on bounded subsets of $[0,T] \times \mathcal{H}$.
    \end{itemize}
    Let $\py$ satisfy the evolution equation
    \begin{equation} \label{a satisfied evo} \py_t = \py_0 + \int_0^t\phi_sds + \int_0^tB_sd\mathcal{W}_s\end{equation}
    $\mathbbm{P}-a.s.$ in $\mathcal{H}$ for all $t \in [0,T]$. Then
    \begin{align*}
        F(t,\py_t) = F(0,\py_0) &+ \int_0^t\left(\partial_tF\left(s, \py_s\right) + D_hF\left(s, \py_s\right)[\phi_s]\right) ds + \int_0^t D_hF\left(s, \py_s\right)[B_s] d\mathcal{W}_s\\
        &+ \frac{1}{2}\int_0^t\sum_{i=1}^\infty D_{hh}F\left(s,\py_s
        \right)[B_s(e_i)][B_s(e_i)] ds 
    \end{align*}
    $\mathbbm{P}-a.s.$ in $\mathcal{K}$ for all $t \in [0,T]$. 
\end{lemma}

Note that in the above, the noise operator is defined on $\mathfrak{U}$ by $D_hF\left(s, \py_s\right)[B_s](e_i) = D_hF\left(s, \py_s\right)[B_s(e_i)]$. We now move towards a proper definition of the Stratonovich integral. Denote by $\mathcal{M}^2_c(\mathcal{H})$ the usual space of continuous, square-integrable martingales in $\mathcal{H}$. As in [\cite{goodair2024stochastic}] Definition 2.18, for $\py \in \mathcal{M}^2_c(\mathcal{H})$ and $Y \in \mathcal{M}^2_c(\R)$ we define the cross-variation $[\py,Y]$ with respect to an orthonormal basis $(a_k)$ of $\mathcal{H}$ by
$$ [\py, Y] := \sum_{k=1}^\infty[\inner{\py}{a_k}_{\mathcal{H}}, Y]a_k$$ $\mathbbm{P}-a.s.$  for the limit taken in $C\left([0,T];\mathcal{H} \right)$, where $[\inner{\py}{a_k}_{\mathcal{H}}, Y]$ is the usual one dimensional cross-variation. This is characterised in [\cite{goodair2024stochastic}] Proposition 2.14, and naturally extends to continuous and square-integrable semi-martingales, which we represent by $\bar{\mathcal{M}}^2_c$. We can now define the Stratonovich integral for locally square-integrable semi-martingales, as required for (\ref{strat ident}).

\begin{definition} \label{strat definition}
Suppose that there exists a sequence of stopping times $(\tau_n) \in [0,T]$ which are $\mathbbm{P}-a.s.$ monotonically increasing and eventually equal to another stopping time $\tau \in [0,T]$ such that:
\begin{enumerate}
    \item \label{item 1}For every $n$, the process $$B^n_{\cdot}:=B_{\cdot}\mathbbm{1}_{\cdot \leq \tau^n}$$
    is progressively measurable and belongs to $L^2\left( \Omega \times [0,T]; \mathscr{L}^2(\mathfrak{U};\mathcal{H})\right)$;
 \item \label{item 2} For every $n$ and $i$, the process 
 $$ B^{\tau_n}_{\cdot}(e_i):= B_{\cdot \wedge \tau^n}(e_i)$$
 belongs to $\bar{\mathcal{M}}^2_c(\mathcal{H})$;
 \item \label{item 3} For every $t \in [0,T]$ the limit $$\sum_{i=1}^\infty [B^{\tau_n}(e_i),W^i]_t$$ is well defined in $L^2\left(\Omega;\mathcal{H}\right)$.
\end{enumerate}
Then the Stratonovich stochastic integral is defined for any $t \in [0,T]$ as $$\int_0^{t \wedge \tau} B_s \circ d\mathcal{W}_s := \lim_{n \rightarrow \infty}\left(\sum_{i=1}^\infty \left(\int_0^tB^n_{s}(e_i) dW^i_s + \frac{1}{2}[B^{\tau_n}(e_i),W^i]_t\right)\right)$$ where the limit is taken $\mathbbm{P}-a.s.$ in $\mathcal{H}$, and the infinite sum in $L^2\left(\Omega;\mathcal{H}\right)$.
\end{definition}

Whilst the stopping time $\tau$ could simply be $T$, for us it will be the local time of existence of the SPDE as in Theorem \ref{main theorem}. We conclude the preliminaries with a lemma to facilitate computations of the cross-variation in the proof of Theorem \ref{main theorem}. This can be found in [\cite{goodair2024stochastic}] Lemma 2.6.

\begin{lemma} \label{cross variation convergence}
    Suppose that $(\py^n)$ is a sequence of martingales in $\mathcal{M}^2_c(\mathcal{H})$ which at every time $t \in [0,T]$, converges in $L^2\big(\Omega;\mathcal{H}\big)$ to some $\py_t$. Let $Y \in \mathcal{M}^2_c$. Suppose in addition that at every time $t \in [0,T]$, the sequence $\left([\py^n,Y]_t \right)$ converges to some $L_t$ in $L^1\big(\Omega;\mathcal{H}\big)$ where $L$ is a continuous, adapted process and for every basis vector $a_k$, $\inner{L}{a_k}_{\mathcal{H}}$ is of bounded variation $\mathbbm{P}-a.s.$. Then $\py \in\mathcal{M}^2_c(\mathcal{H})$ and $[\py,Y]$ is indistinguishable from $L$. 
\end{lemma}

\section{Proof of the Main Result} \label{section proof of main}

\begin{proof}[Proof of Theorem \ref{main theorem}:]
To begin the proof, we verify that the suggested It\^{o}-Stratonovich corrector in (\ref{converted identity}) is well defined. Indeed,
\begin{align} \nonumber
    \int_0^{\tau} \sum_{i=1}^\infty\norm{ D_u\mathcal{G}_i\left(s, \sy_s\right)\left[\mathcal{G}_i(s,\sy_s)\right]}_Uds &\leq \int_0^{\tau} \sum_{i=1}^\infty\norm{ D_u\mathcal{G}_i\left(s, \sy_s\right)}_{\mathscr{L}(H;U)}\norm{\mathcal{G}_i(s,\sy_s)}_Hds\\
    &\leq \int_0^{\tau} \sum_{i=1}^\infty c_i^2\left( 1 + \norm{\sy_s}_H^q\right)\left(1 + \norm{\sy_s}_V \right)ds \label{numbered}
\end{align}
which is finite $\mathbbm{P}-a.s.$ due to the square summability of $(c_i)$ and the regularity of $\sy$. Next we introduce the sequence of stopping times $(\tau^n)$ localising the Stratonovich integral of (\ref{strat ident}) in order to define it in $X$. Let
\begin{equation} 
    \tau_n:= \tau \wedge \inf\left\{s \geq 0: \sup_{r \in [0,s]}\norm{\sy_r}_H^2 + \int_0^{s \wedge \tau}\norm{\sy_r}_{V}^2 dr \geq n \right\}
\end{equation}
with the convention that the infimum of the empty set is infinite. Then $(\tau_n)$ are $\mathbbm{P}-a.s.$ monotonically increasing and eventually equal to $\tau$. Item \ref{item 1} of Definition \ref{strat definition} is satisfied as, labelling $\mathcal{G}\left(\cdot, \sy_{\cdot}\right)^n := \mathcal{G}\left(\cdot, \sy_{\cdot}\right)\mathbbm{1}_{\cdot \leq \tau_n}$,
\begin{align*}\mathbbm{E}\left(\int_0^T\sum_{i=1}^\infty\norm{\mathcal{G}_i(s,\sy_s)^n\mathbbm{}}_H^2ds\right) &\leq \mathbbm{E}\left(\int_0^T\sum_{i=1}^\infty c_i^2\left(1 + \norm{\sy_s}_V\right)^2\mathbbm{1}_{s \leq \tau_n} ds\right)\\ &\leq 2\left[\sum_{i=1}^\infty c_i^2\right]\mathbbm{E}\left(\int_0^{\tau_n}\left(1 + \norm{\sy_s}_V^2\right) ds\right)
\end{align*}
which is finite due to control from the stopping time. Moreover, the process is progressively measurable as $\mathcal{G}\left(\cdot, \sy_{\cdot}\right)^n = \mathcal{G}\left(\cdot, \sy_{\cdot}\mathbbm{1}_{\cdot \leq \tau}\right)\mathbbm{1}_{\cdot \leq \tau_n}$ such that measurability of $\mathcal{G}|_V$ and the assumed progressive measurability of $\sy_{\cdot}\mathbbm{1}_{\cdot \leq \tau}$ is sufficient. For item \ref{item 2} of Definition \ref{strat definition} we look to the evolution equation satisfied by $\mathcal{G}_i\left(\cdot, \sy_{\cdot}\right)^{\tau_n}:= \mathcal{G}_i\left(\cdot \wedge \tau_n, \sy_{\cdot \wedge \tau_n}\right)$. Observe that to stop the integrals in (\ref{converted identity}) one can instead introduce $\mathbbm{1}_{\cdot \leq \tau}$ into the integrands, so that (\ref{converted identity}) matches the form (\ref{a satisfied evo}). Therefore we may apply Lemma \ref{Ito formula} to deduce that
\begin{align*}
    \mathcal{G}_i\left(t,\sy_{t}\right) = \mathcal{G}_i\left(0,\sy_0\right) &+ \int_0^{t} D_u\mathcal{G}_i\left(s,\sy_s\right)\left[\left(\mathcal{A}\left(s,\sy_s\right) + \frac{1}{2}\sum_{i=1}^\infty D_u\mathcal{G}_i\left(s, \sy_s\right)\left[\mathcal{G}_i(s,\sy_s)\right] \right)\mathbbm{1}_{s \leq \tau}\right]ds\\ &+ \int_0^t \partial_t\mathcal{G}_i\left(s,\sy_s\right)ds + \int_0^{t}D_u\mathcal{G}_i\left(s,\sy_s\right)\left[\mathcal{G}\left(s, \sy_s\right)\mathbbm{1}_{s \leq \tau}\right] d\mathcal{W}_s\\ &+ \frac{1}{2}\int_0^t\sum_{i=1}^\infty D_{uu}\mathcal{G}_i\left(s, \sy_s\right)\left[\mathcal{G}_i(s,\sy_s)\mathbbm{1}_{\cdot \leq \tau} \right]\left[\mathcal{G}_i(s,\sy_s) \mathbbm{1}_{\cdot \leq \tau}\right] ds
\end{align*}
holds $\mathbbm{P}-a.s.$ in $X$ for all $t \in [0,T]$. As the time integrals are well defined in $X$, it is clear that they are of finite variation. We confirm that the stochastic integral is a genuine square-integrable martingale. As $D_u\mathcal{G}_i\left(s,\sy_s\right)$ is linear and $\tau_n \leq \tau$ $\mathbbm{P}-a.s.$, we may rewrite
$$\int_0^{t \wedge \tau_n}D_u\mathcal{G}_i\left(s,\sy_s\right)\left[\mathcal{G}\left(s, \sy_s\right)\mathbbm{1}_{s \leq \tau}\right] d\mathcal{W}_s = \int_0^{t \wedge \tau_n}D_u\mathcal{G}_i\left(s,\sy_s\right)\left[\mathcal{G}\left(s, \sy_s\right)\right] d\mathcal{W}_s.$$
Square-integrability, and thus that $\mathcal{G}_i\left(\cdot, \sy_{\cdot}\right)^{\tau_n} \in \bar{\mathcal{M}}^2_c(X)$, follows similarly to (\ref{numbered}) as
\begin{align} \nonumber
    \mathbbm{E}\left(\int_0^{\tau_n}\sum_{j=1}^\infty \norm{D_u\mathcal{G}_i\left(s, \sy_s\right)\left[\mathcal{G}_j(s,\sy_s)\right]}_X^2\right) ds &\leq 
    c\mathbbm{E}\left(\int_0^{\tau_n}\sum_{j=1}^\infty \norm{D_u\mathcal{G}_i\left(s, \sy_s\right)\left[\mathcal{G}_j(s,\sy_s)\right]}_U^2 ds\right)\\
    &\leq c\mathbbm{E}\left(\int_0^{\tau_n} \sum_{j=1}^\infty c_i^2c_j^2\left( 1 + \norm{\sy_s}_H^{q}\right)^2\left(1 + \norm{\sy_s}_V \right)^2ds\right) \nonumber
\end{align}
which is again finite due to control from the stopping time. We now examine the cross-variation,
$$\left[\mathcal{G}_i\left(\cdot, \sy_{\cdot}\right)^{\tau_n}, W^i \right]_t = \left[\int_0^{\cdot \wedge \tau_n}D_u\mathcal{G}_i\left(s,\sy_s\right)\left[\mathcal{G}\left(s, \sy_s\right)\right] d\mathcal{W}_s, W^i \right]_t$$
and employing Lemma \ref{cross variation convergence} alongside the definition of the integral as an $L^2\left(\Omega;X\right)$ limit, then
\begin{align*}
    \left[\int_0^{\cdot \wedge \tau_n}D_u\mathcal{G}_i\left(s,\sy_s\right)\left[\mathcal{G}\left(s, \sy_s\right)\right] d\mathcal{W}_s, W^i \right]_t &= \left[\sum_{j=1}^\infty\int_0^{\cdot \wedge \tau_n}D_u\mathcal{G}_i\left(s,\sy_s\right)\left[\mathcal{G}_j\left(s, \sy_s\right)\right] dW^j_s, W^i \right]_t\\
    &= \lim_{m \rightarrow \infty}\left[\sum_{j=1}^m\int_0^{\cdot \wedge \tau_n}D_u\mathcal{G}_i\left(s,\sy_s\right)\left[\mathcal{G}_j\left(s, \sy_s\right)\right] dW^j_s, W^i \right]_t
\end{align*}
for the limit in $L^1\left(\Omega;X\right)$, should it exist and satisfy the conditions of Lemma \ref{cross variation convergence}. To inspect this we introduce an orthonormal basis $(a_k)$ of $X$, so that by definition the above cross-variation is given by 
$$\sum_{k=1}^\infty \left[\inner{\sum_{j=1}^m\int_0^{\cdot \wedge \tau_n}D_u\mathcal{G}_i\left(s,\sy_s\right)\left[\mathcal{G}_j\left(s, \sy_s\right)\right] dW^j_s}{a_k}_X, W^i \right]_ta_k $$
or equivalently
$$\sum_{k=1}^\infty \left[\sum_{j=1}^m\int_0^{\cdot \wedge \tau_n}\inner{D_u\mathcal{G}_i\left(s,\sy_s\right)\left[\mathcal{G}_j\left(s, \sy_s\right)\right]}{a_k}_X dW^j_s, W^i \right]_ta_k.$$
From classical finite dimensional theory, due to the independence of the Brownian Motions then for $m \geq i$, this reduces to
\begin{align*}
    \sum_{k=1}^\infty \left(\int_0^{t \wedge \tau_n}\inner{D_u\mathcal{G}_i\left(s,\sy_s\right)\left[\mathcal{G}_i\left(s, \sy_s\right)\right]}{a_k}_X ds \right)a_k = \int_0^{t \wedge \tau_n} D_u\mathcal{G}_i\left(s,\sy_s\right)\left[\mathcal{G}_i\left(s, \sy_s\right)\right] ds
\end{align*}
which satisfies the conditions of Lemma \ref{cross variation convergence}. Therefore
$$\sum_{i=1}^\infty \left[\mathcal{G}_i\left(\cdot, \sy_{\cdot}\right)^{\tau_n}, W^i_{\cdot} \right]_t = \sum_{i=1}^\infty\int_0^{t \wedge \tau_n} D_u\mathcal{G}_i\left(s,\sy_s\right)\left[\mathcal{G}_i\left(s, \sy_s\right)\right] ds  $$
which we must justify converges in $L^2\left(\Omega;X\right)$. Once more this is similar to (\ref{numbered}), as
\begin{align*}
    &\sum_{i=1}^\infty \left(\mathbbm{E}\left(\norm{\int_0^{t \wedge \tau_n} D_u\mathcal{G}_i\left(s,\sy_s\right)\left[\mathcal{G}_i\left(s, \sy_s\right)\right] ds}_X^2 \right)\right)^{\frac{1}{2}}\\ & \qquad \qquad \qquad \leq c\sum_{i=1}^\infty \left(\mathbbm{E}\left(\int_0^{t \wedge \tau_n}\norm{\ D_u\mathcal{G}_i\left(s,\sy_s\right)\left[\mathcal{G}_i\left(s, \sy_s\right)\right]}_U^2 ds \right)\right)^{\frac{1}{2}}\\
    & \qquad \qquad \qquad \leq c\sum_{i=1}^\infty \left(\mathbbm{E}\left(\int_0^{t \wedge \tau_n} c_i^4\left(1 + \norm{\sy_s}^{2q}_H \right)\left(1 + \norm{\sy_s}^{2}_V \right) ds \right)\right)^{\frac{1}{2}}\\
    & \qquad \qquad \qquad \leq \tilde{c}_n\left(\sum_{i=1}^\infty c_i^2\right) \left(\mathbbm{E}\left(\int_0^{t \wedge \tau_n} \left(1 + \norm{\sy_s}^{2}_V \right) ds \right)\right)^{\frac{1}{2}}
\end{align*}
where $\tilde{c}_n$ is some constant dependent on $n$ owing to the bounds from $\tau_n$, which is finite as established. Thus, the Stratonovich integral appearing in (\ref{strat ident}) is well defined by
    $$\int_0^{t \wedge \tau}\mathcal{G}_i\left(s, \sy_s\right) \circ d\mathcal{W}_s = \lim_{n \rightarrow \infty} \left(\sum_{i=1}^\infty\left(\int_0^t \mathcal{G}\left(s, \sy_s\right)^ndW^i_s + \frac{1}{2}\int_0^{t \wedge \tau_n} D_u\mathcal{G}_i\left(s,\sy_s\right)\left[\mathcal{G}_i\left(s, \sy_s\right)\right] ds \right) \right)$$
for the limit taken $\mathbbm{P}-a.s.$ in $X$, and the infinite sum in $L^2\left(\Omega;X \right)$. The first term is precisely the definition of the It\^{o} integral
$$\int_0^{t \wedge \tau}\mathcal{G}\left(s, \sy_s\right) d\mathcal{W}_s $$
appearing in (\ref{converted identity}). Comparing terms in (\ref{converted identity}) and (\ref{strat ident}), it only remains to show that
$$\lim_{n \rightarrow \infty} \left(\sum_{i=1}^\infty\int_0^{t \wedge \tau_n} D_u\mathcal{G}_i\left(s,\sy_s\right)\left[\mathcal{G}_i\left(s, \sy_s\right)\right] ds \right) = \int_0^{t \wedge \tau}\sum_{i=1}^\infty  D_u\mathcal{G}_i\left(s,\sy_s\right)\left[\mathcal{G}_i\left(s, \sy_s\right)\right] ds.$$
To see this we can extract a $\mathbbm{P}-a.s.$ convergent subsequence of the partial sums converging in $L^2\left(\Omega;X\right)$, which we know to be
$$ \int_0^{t \wedge \tau_n}\sum_{i=1}^\infty  D_u\mathcal{G}_i\left(s,\sy_s\right)\left[\mathcal{G}_i\left(s, \sy_s\right)\right] ds$$
from (\ref{numbered}). It is immediate that
$$\lim_{n \rightarrow \infty} \left(\int_0^{t \wedge \tau_n}\sum_{i=1}^\infty D_u\mathcal{G}_i\left(s,\sy_s\right)\left[\mathcal{G}_i\left(s, \sy_s\right)\right] ds \right) = \int_0^{t \wedge \tau}\sum_{i=1}^\infty  D_u\mathcal{G}_i\left(s,\sy_s\right)\left[\mathcal{G}_i\left(s, \sy_s\right)\right] ds$$
which concludes the proof.

\end{proof}

\section{Applications} \label{section applications}

Applications of Theorem \ref{main theorem} are given in this section. The simple case where the noise is linear is given in Subsection \ref{subs linear op} and is shown to agree with the heuristic given in the introduction. We provide an existence result for the Navier-Stokes equations under a time-dependent transport noise in Subection \ref{subs time dep}. A nonlinear transport noise is considered in Subsection \ref{subs nonlinear trans}.

\subsection{Linear Operators} \label{subs linear op}

\begin{lemma} \label{linear lemma}
    In addition to the assumptions of Theorem \ref{main theorem}, suppose that for every $i \in \N$ and $s \in [0,T]$ that $\mathcal{G}_i(s, \cdot)$ is linear. Then the identity (\ref{converted identity}) reduces to
    \begin{equation} \nonumber    \sy_{t} = \sy_0 + \int_0^{t\wedge \tau} \left(\mathcal{A}\left(s,\sy_s\right) + \frac{1}{2}\sum_{i=1}^\infty \mathcal{G}_i\left(s, \mathcal{G}_i(s,\sy_s)\right)\right)ds + \int_0^{t \wedge \tau}\mathcal{G}\left(s, \sy_s\right) d\mathcal{W}_s
\end{equation}
and the same conclusion holds.
\end{lemma}

\begin{proof}
    This is an immediate consequence of the fact that for $F: U \rightarrow X$ linear and $\phi,\psi \in U$ that $D_uF(\phi)[\psi] = F(\psi)$.
\end{proof}

By Lemma \ref{linear lemma}, when $\mathcal{G}_i$ is further time-independent we recover the form given by the heuristic with transport noise stated in the introduction.  

\subsection{Time-Dependent Transport Noise} \label{subs time dep}

Our next application is the case of a time-dependent transport noise, viewed in an existence result for the Navier-Stokes equations presented now on the $N-$dimensional torus $\T^N$. For the functional analytic framework we recall that any function $f \in L^2(\T^N;\R^N)$ admits the representation \begin{equation} \label{fourier rep}f(x) = \sum_{k \in \mathbb{Z}^N}f_ke^{ik\cdot x}\end{equation} whereby each $f_k \in \mathbb{C}^N$ is such that $f_k = \overbar{f_{-k}}$ and the infinite sum is defined as a limit in $L^2(\T^N;\R^N)$, see e.g. [\cite{robinson2016three}] Subsection 1.5 for details. We then define $L^2_{\sigma}$ as the subset of $L^2(\T^N;\R^N)$ of zero-mean functions $f$ whereby for all $k \in \mathbbm{Z}^N$, $k \cdot f_k = 0$ with $f_k$ as in (\ref{fourier rep}). For general $m \in \N$ we introduce $W^{m,2}_{\sigma}$ as the intersection of $W^{m,2}(\T^N;\R^N)$ respectively with $L^2_{\sigma}$, $W^{0,2}_{\sigma}: = L^2_{\sigma}$ and $W^{-1,2}_{\sigma}:= \left(W^{1,2}_{\sigma}\right)^*$, the dual space. Set $\mathcal{P}$, the Leray Projector, as the orthogonal projection in $L^2(\T^N;\R^N)$ onto $L^2_{\sigma}$. We continue to use $\mathcal{L}$ from Subsection \ref{subs motive}, defined for sufficiently regular vector fields $f,g$ by $\mathcal{L}_{f}g = \sum_{j=1}^Nf^j\partial_j g$. The noise operator that we consider, following its introduction in [\cite{holm2015variational}] but now allowing for the spatial correlation functions $(\xi_i)$ to evolve in time, is defined along the basis components $e_i$ of the noise space $\mathfrak{U}$ by
$$B_i(s,f) := \sum_{j=1}^N\left( \xi_i(s)^j\partial_jf + f^j \nabla \xi_i(s)^j\right).$$
In the statement of the result, we shall denote $W^{k,\infty}(\T^N;\R^N)$ by $W^{k,\infty}$ for simplicity.

\begin{lemma}
        Let $1 \leq m \in \N$, $u_0: \Omega \rightarrow W^{m,2}_{\sigma}$ be $\mathcal{F}_0-$measurable and $\xi_i \in C^1\left( [0,T] ; W^{m+2,\infty}\right)$ with $\xi_i:[0,T] \rightarrow L^2_{\sigma} \cap W^{m+3,\infty}$ and such that $\sum_{i=1}^\infty \sup_{t \in [0,T]}\left(  \norm{\xi_i(t)}_{W^{m+2,\infty}}^2\right) < \infty$. Then there exists a pair $(u,\tau)$ where:
        \begin{enumerate}
            \item $\tau \in (0,T]$ $\mathbbm{P}-a.s.$ is a stopping time;
            \item $u$ is a process such that $u_{\cdot}\mathbbm{1}_{\cdot \leq \tau}$ is progressively measurable in $W^{m+1,2}_{\sigma}$ whilst for $\mathbbm{P}-a.e.$ $\omega$, $u_{\cdot}(\omega) \in C\left([0,T];W^{m,2}_{\sigma}\right)$ and $u_{\cdot}(\omega)\mathbbm{1}_{\cdot \leq \tau(\omega)} \in L^2\left([0,T];W^{m+1,2}_{\sigma}\right)$;
            \item The identity $$  u_t = u_0 - \int_0^{t \wedge \tau}\mathcal{P}\mathcal{L}_{u_s}u_s\ ds + \int_0^{t\wedge \tau} \Delta u_s\, ds  - \int_0^{t\wedge \tau} \mathcal{P}B(s,u_s) \circ d\mathcal{W}_s$$
holds $\mathbbm{P}-a.s.$ in $W^{m-2,2}_{\sigma}$ for all $t \in [0,T]$.
        \end{enumerate}
Moreover if $N=2$ one can choose $\tau := T$.

\end{lemma}

\begin{proof}
As each $\mathcal{P}B_i(s,\cdot)$ is linear, then by applying Lemma \ref{linear lemma} for the spaces
$$V:= W^{m+1,2}_{\sigma}, \qquad H:= W^{m,2}_{\sigma}, \qquad  U:= W^{m-1,2}_{\sigma}, \qquad X:= W^{m-2,2}_{\sigma} $$
we see that to verify the Stratonovich identity, it is sufficient to satisfy the It\^{o} form
$$  u_t = u_0 - \int_0^{t \wedge \tau}\mathcal{P}\mathcal{L}_{u_s}u_s\ ds + \int_0^{t\wedge \tau} \Delta u_s\, ds + \frac{1}{2}\int_0^{t \wedge \tau} \sum_{i=1}^\infty \mathcal{P}B_i\left(s, \mathcal{P}B_i(s,u_s)\right)ds - \int_0^{t\wedge \tau} \mathcal{P}B(s,u_s) d\mathcal{W}_s$$
in $W^{m-1,2}_{\sigma}$. Note that in addition to their spatial smoothness and summability, each $\xi_i$ is continuously differentiable in time to facilitate Lemma \ref{linear lemma}. When the $(\xi_i)$ are time independent, global existence of strong solutions of the It\^{o} form in 2D was proven in [\cite{goodair2024weak}] Theorem 5.4, local existence of strong solutions in 3D in [\cite{goodair2022existence1}] also Theorem 5.4, and the propagation of regularity of these solutions in [\cite{goodair2024high}] Proposition 3.7. In fact in all cases, the results were deduced by applying an abstract criterion in each paper: therefore, one only needs to inspect how the introduction of time-dependence into the $(\xi_i)$ affects a verification of the assumptions in those papers. Identical bounds in terms of $\xi_i(t)$ are obtained, and due to the condition $\sum_{i=1}^\infty \sup_{t \in [0,T]}\left(  \norm{\xi_i(t)}_{W^{m+2,\infty}}^2\right) < \infty$ they hold verbatim and we conclude the proof.
\end{proof}

\subsection{Nonlinear Transport Noise} \label{subs nonlinear trans}

The purpose of this subsection is to elucidate computations regarding a nonlinear transport noise as motivated in the introduction. Given the many possible ways of incorporating such a noise into an equation, we cannot expect to capture all rigorous possibilities of application. Therefore, we provide only a formal result by computing the expression given in Theorem \ref{main theorem}; in any specific application, Theorem \ref{main theorem} provides the tools to make this rigorous.

\begin{lemma}
Consider $\sy$ on some domain $\mathscr{O} \subset \R^n$, $\sy: [0,T] \times \Omega \times \mathscr{O} \rightarrow \R^d$, understood as a function space valued process $\sy: [0,T] \times \Omega \rightarrow H$,  specified by the Stratonovich SPDE
    \begin{equation} \label{formal strato}
        \sy_{t} = \sy_0 + \int_0^{t} \mathcal{A}\left(s,\sy_s\right)ds + \sum_{i=1}^\infty\int_0^{t} \mathcal{L}_{\xi_i}F(\sy_s) \circ dW^i_s
    \end{equation}
    where $H \hookrightarrow U$ are Hilbert Spaces and $F: H \rightarrow U$ is continuous with sufficiently regular Fr\'{e}chet derivatives $D_hF$, $D_{hh}F$. Then (\ref{formal strato}) has corresponding It\^{o} form
\begin{equation} \label{formal ito}
        \sy_{t} = \sy_0 + \int_0^{t} \left(\mathcal{A}\left(s,\sy_s\right) + \frac{1}{2}\sum_{i=1}^\infty \mathcal{L}_{\xi_i}\left(D_hF(\sy_s)\left[ \mathcal{L}_{\xi_i}F(\sy_s) \right] \right) \right)ds + \sum_{i=1}^\infty\int_0^{t} \mathcal{L}_{\xi_i}F(\sy_s) dW^i_s.
    \end{equation}
Suppose that $F$ is defined by a differentiable function $f: \R^d \rightarrow \R^d$ through $F(\psi)(x) = f(\psi(x))$. Then (\ref{formal ito}) reduces to  
\begin{equation} \nonumber
        \sy_{t} = \sy_0 + \int_0^{t} \left(\mathcal{A}\left(s,\sy_s\right) + \frac{1}{2}\sum_{i=1}^\infty \mathcal{L}_{\xi_i}\left(f'(\sy_s)^2 \mathcal{L}_{\xi_i}(\sy_s)\right) \right)ds + \sum_{i=1}^\infty\int_0^{t}f'(\sy_s) \mathcal{L}_{\xi_i}(\sy_s) dW^i_s.
    \end{equation}
    
\end{lemma}

\begin{proof}
    We verify the It\^{o}-Stratonovich corrector, given from Theorem \ref{main theorem}, for each $i$ in the summand, by $D_h\left( \mathcal{L}_{\xi_i}F\right)(\sy_s)\left[\mathcal{L}_{\xi_i}F(\sy_s) \right]$. By the chain rule for Fr\'{e}chet derivatives, this is
$$D_h\mathcal{L}_{\xi_i}\left(F(\sy_s)\right)\left[D_hF(\sy_s)\left[\mathcal{L}_{\xi_i}F(\sy_s) \right] \right] $$
    and as $\mathcal{L}_{\xi_i}$ is linear, it further simplifies to
$$\mathcal{L}_{\xi_i}\left(D_hF(\sy_s)\left[\mathcal{L}_{\xi_i}F(\sy_s) \right] \right) $$
    as required. The second part of the lemma follows from the facts that $D_hF(\phi)[\psi] = f'(\phi)\psi$ and $\mathcal{L}_{\xi_i}\left(f(\psi)\right) = f'(\psi)\mathcal{L}_{\xi_i}\psi$.
\end{proof}

The second part of the lemma recovers the conversion given in [\cite{flandoli20242d}] page 6, where it is further shown that if $\xi_i$ is divergence-free and $d=1$ then
$$\mathcal{L}_{\xi_i}\left(f'(\sy_s)^2 \mathcal{L}_{\xi_i}(\sy_s)\right) = \textnormal{div}\left(f'(\sy_s)^2(\xi_i \otimes \xi_i)\nabla \sy_s \right). $$


\addcontentsline{toc}{section}{References}
\bibliographystyle{newthing}
\bibliography{myBibby}

\end{document}